\documentclass[12pt]{amsart}
\usepackage{amsmath,amsthm,amssymb}

\usepackage{cite,float}

\newcommand{\disc}{\mathbb{D}}
\DeclareMathOperator{\RE}{Re}

\numberwithin{equation}{section}
\newtheorem{theorem}{Theorem}[section]
\newtheorem{lemma}[theorem]{Lemma}
\newtheorem{corollary}[theorem]{Corollary}
\theoremstyle{remark}
\newtheorem{remark}[theorem]{Remark}

\makeatletter
\@namedef{subjclassname@2010}{%
  \textup{2010} Mathematics Subject Classification}
\makeatother

\begin{document}
 	\title{MARX-STROHH\"{A}CKER THEOREM FOR MULTIVALENT FUNCTIONS}
 	
 	\author[P. Gupta]{Prachi Gupta}
 	\address{Department of Mathematics, University of Delhi, Delhi-110007, India}
 	\email{prachigupta161@gmail.com}

 	\author[S. Nagpal]{Sumit Nagpal}
 	\address{Department of Mathematics, Ramanujan College, University of Delhi,
 		Delhi--110019, India}
 	\email{sumitnagpal.du@gmail.com }
 	
 	\author[V. Ravichandran]{V. Ravichandran}
 	
 	\address{Department of Mathematics,
 		National Institute of Technology,
 		Tiruchirappalli-620015, India}
 	\email{ravic@nitt.edu; vravi68@gmail.com}
 	
 	\keywords{differential subordination, fixed initial coefficient, convex and
 			starlike functions, multivalent function.}
 	\subjclass[2010]{Primary 30C80; Secondary 30C45}
 	
 	\maketitle
 	\begin{abstract}
Some differential implications of classical Marx-Strohh\"acker theorem are extended for multivalent functions. These results are also generalized for functions with fixed second coefficient by using the theory of first order differential subordination which in turn, corrects the results of Selvaraj and Stelin [On multivalent functions associated with fixed second coefficient and the principle of subordination, Int. J.  Math. Anal. {\bf 9} (2015), no.~18, 883--895].
 	\end{abstract}
 	
	\section{Introduction}
 One of the classical result in univalent function theory is the Marx-Strohh\"{a}cker theorem \cite {MARK,STRO} which connects convex and starlike functions. Although the original proof of theorem was complicated, Miller and Mocanu \cite[Section 2.6, p.\ 56]{MM} gave simple algebraic proof using the technique of differential subordination. In 2017, Nunokawa \emph{et al.} \cite{NUNO2} extended some of these results for multivalent functions. In this paper, we extend some other forms of Marx-Strohh\"{a}cker type results for multivalent functions. Moreover, these results are generalized using the theory of first-order differential subordination formulated by Ali \emph{et al.} \cite{ALI} for functions with fixed initial coefficient.

For $n\in \mathbb{N}$ and $a\in \mathbb{C}$, let $\mathcal{H}[a,n]$ denotes the class of analytic functions $f$ defined in the open unit disk $\mathbb{D}:=\{z\in\mathbb{C}:|z|<1\}$ of the form
\[f(z)=a+a_nz^n+a_{n+1}z^{n+1}+\cdots.\]
Also, let $\mathcal{A}_p$ ($p\in \mathbb{N}$) be its subclass consisting of normalized functions $f$ of the form
\[f(z)=z^p+a_{p+1}z^{p+1}+a_{p+2}z^{p+2}+\cdots.\]
Set $\mathcal{A}:=\mathcal{A}_1$. The subclass of $\mathcal{A}$ consisting of univalent functions is denoted by $\mathcal{S}$. A function $f\in \mathcal{A}_p$ is said to be $p$-valent convex of order $\alpha$ (or $p$-valent starlike of order $\alpha$) in $\mathbb{D}$ ($0\leq \alpha<p$) if
\[\RE\left(1+\frac{zf''(z)}{f'(z)}\right)>\alpha\quad \left(\mbox{or}\quad \RE\left(\frac{zf'(z)}{f(z)}\right)>\alpha\right)\]
for all $z\in \mathbb{D}$. More details of $p$-valent starlike and $p$-valent convex functions can be found in \cite{NUNO1,NUNO2,NUNO3,srivastava}. For $f\in\mathcal{A}$, Marx-Strohh\"{a}cker theorem asserts that, for all $z\in\disc$, \begin{equation}\label{eq1}
\RE\left(1+\frac{zf''(z)}{f'(z)}\right)>0\,  \Rightarrow\, \RE\left(\frac{zf'(z)}{f(z)}\right)>\frac{1}{2}\, \Rightarrow\,  \RE\left(\frac{f(z)}{z}\right)>\frac{1}{2}\,\, .
\end{equation}
Nunokawa \emph{et al.} \cite{NUNO2} extended these differential implications for multivalent functions $f\in \mathcal{A}_p$ ($p\geq 2$) by finding $\beta$ and $\gamma$ such that
\begin{equation}\label{eq2}
\RE\left(1+\frac{zf''(z)}{f'(z)}\right)>\alpha\,  \Rightarrow\, \RE\left(\frac{zf'(z)}{f(z)}\right)>\beta\, \Rightarrow\,  \RE\left(\frac{f(z)}{z^p}\right)>\gamma \, \,  .
\end{equation}
There are two more differential implications in Marx-Strohh\"{a}cker theorem:
 \begin{equation}\label{eq3}
\RE\left(1+\frac{zf''(z)}{f'(z)}\right)>0\,  \Rightarrow\, \RE\sqrt{f'(z)}>\frac{1}{2}\, \Rightarrow\,  \RE\left(\frac{f(z)}{z}\right)>\frac{1}{2}\,\,
\end{equation}
for a function $f\in \mathcal{A}$. These implications are extended for the class of multivalent functions $\mathcal{A}_p$ ($p\geq 2$) in Section 2 using the following lemma due to Miller and Mocanu.

 \begin{lemma}\cite{MM}\label{MM}
 Let $\Omega\subseteq\mathbb{C}$ and $\psi:\mathbb{C}^2\to\mathbb{C}$ satisfies the admissibility condition $\psi(i\rho,\sigma)\not\in \Omega$ whenever $\sigma\leq -n (1+\rho)^2/2$
 where $\rho \in \mathbb{R}$ and $n\in\mathbb{N}$. If $p\in\mathcal{H}[1,n]$ and $\psi(p(z),zp'(z))\in \Omega$ for $z\in\mathbb{D}$, then $\RE p(z)>0$ for all $z\in\mathbb{D}$.
 \end{lemma}

The second coefficient of normalized univalent functions play a vital role in shaping the geometric as well as analytic properties of functions. Ali \emph{et al.} \cite{ALI} reformulated the Miller and Mocanu's differential subordination theory \cite{M1,M2} for analytic functions with fixed preassigned initial coefficient. Several authors have used this technique to investigate the properties of normalized univalent functions with fixed second coefficient.

Let $\mathcal{H}_\zeta[a,n]$ denote the class of analytic functions $p\in \mathcal{H}[a,n]$ of the form $p(z)=a+\zeta z^n+a_{n+1}z^{n+1}+\cdots$, where $\zeta\in\mathbb{C}$ is fixed. Without loss of generality, we assume that $\zeta$ is a positive real number. Similarly, let $\mathcal{A}_{p,b}$ denote the class of all functions $f \in \mathcal{A}_p$ of the form $f(z)=z^p+bz^{p+1}+a_{p+2}z^{p+2}+\cdots$ where $b$ is a fixed non-negative real number. Nagpal and Ravichandran \cite{NAGPAL} proved \eqref{eq1} for functions $f\in \mathcal{A}_{1,b}$. In Section 3, we extend differential implications \eqref{eq2} and results obtained in Section 2 for multivalent functions $f\in \mathcal{A}_{p,b}$. It is worth to note that Selvaraj and Stelin \cite{SELVARAJ} proved the similar results. However, there was a minor error in their proofs. The same has been highlighted in the paper. The following lemma will be need in our investigation.
\begin{lemma}\cite{ALI}\label{ALI}
 Let $\Omega$ be a set in $\mathbb{C}$. Let $\psi:\mathbb{C}^2\to\mathbb{C}$ satisfies the admissibility condition
 \[\psi(i\rho,\sigma)\not\in \Omega \quad \mbox{whenever}\quad \sigma\leq -\frac{2}{2+\zeta}(1+\rho^2)\]
 where $\rho \in \mathbb{R}$, $n\in\mathbb{N}$ and $0<\zeta\leq 2$. If $p\in\mathcal{H}_\zeta[1,n]$ and $\psi(p(z),zp'(z))\in \Omega$ for $z\in\mathbb{D}$, then $\RE p(z)>0$ for all $z\in\mathbb{D}$.
 \end{lemma}

\section{Marx-Strohh\"{a}cker Type Results}
In this section, the differential implications of the form \eqref{eq3} are extended for multivalent functions $f\in \mathcal{A}_p$ ($p\geq 2$) by finding $\beta$ and $\gamma$ such that
 \begin{equation}\label{eq4}
\RE\left(1+\frac{zf''(z)}{f'(z)}\right)>\alpha\,  \Rightarrow\, \RE\sqrt{\frac{f'(z)}{pz^{p-1}}}>\beta\, \Rightarrow\,  \RE\left(\frac{f(z)}{z^p}\right)>\gamma\,\,   (z\in\disc).
\end{equation}
Whenever we talk about square-root of a function, we choose the branch such that $\sqrt{1}=1$.
To prove our results, let us prove the following simple lemma. It is worth to remark that Lemma \ref{lem} can be applied to derive the results obtained in \cite{NUNO2} with simple and less-computational proofs.

\begin{lemma}\label{lem}
If $0\leq b< a$, then the  continuous function $\phi:[0,\infty)\to \mathbb{R}$ defined by
\begin{equation}\label{eq}
\phi(x)\equiv\phi(x,a,b)=\frac{1+x}{(a-b)^2 x+b^2}
\end{equation}
satisfies
\[\min_{x\in [0,\infty)} \phi(x)=\begin{cases}
	\displaystyle\frac{1}{(a-b)^2}, & 0 \leq b\leq \displaystyle\frac{a}{2},\\
	\displaystyle\frac{1}{b^2}, & \displaystyle\frac{a}{2}\leq b<a.
	\end{cases}\]
\end{lemma}

\begin{proof}
Note that
\[\phi'(x)=\frac{a(2b-a)}{((a-b)^2 x+b^2)^2}.\]
\quad Case 1: If $0\leq b\leq a/2$, then $\phi'(x)\leq 0$ so that $\phi$ is a decreasing function of $x$ and the minimum value of $\phi$ is attained at infinity. Thus
\[\min_{x\in [0,\infty)}\phi(x)=\underset{x\rightarrow\infty}{\lim}\phi(x)=\displaystyle\frac{1}{(a-b)^2}.\]
\quad Case 2: If $a/2\leq b<a$, then $\phi'(x)\geq 0$ and hence $\phi$ is an increasing function of $x$ so that the minimum value of $\phi$ is attained at the origin. Therefore
\[\min_{x\in [0,\infty)} \phi(x)=\phi(0)=\displaystyle\frac{1}{b^2}.\qedhere\]
\end{proof}

Given $0\leq \beta<1$, the problem of finding $\alpha$ such that the first differential implication of \eqref{eq4} holds for multivalent functions is investigated in the following theorem.

\begin{theorem}\label{th2.2}
For $p\in\mathbb{N}$ and $0\leq \beta<1$, let
	\begin{align}
	\alpha=\alpha(\beta,p)= \begin{cases}
	p-\displaystyle\frac{\beta}{1-\beta},& 0\leq \beta\leq\displaystyle\frac{1}{2}\\
	\\
	p-\displaystyle\frac{1-\beta}{\beta},& \displaystyle\frac{1}{2}\leq\beta <1.
	\end{cases}
	\end{align}
If $f\in\mathcal{A}_{p}$  is p-valent convex of order $\alpha$, then $$\RE\sqrt{\frac{f'(z)}{pz^{p-1}}}>\beta \quad (z\in \mathbb{D}).$$
\end{theorem}
\begin{proof}
The function $q:\mathbb{D}\to\mathbb{C}$ defined by	\[q(z)=\frac{1}{1-\beta}\left(\sqrt{\frac{f'(z)}{pz^{p-1}}}-\beta\right)\]
is analytic in $\mathbb{D}$ and $q\in\mathcal{H}[1,1].$ A calculation using the equation
\[f'(z)=p((1-\beta)q(z)+\beta)^2z^{p-1}\]
shows that
\[	1+\frac{zf''(z)}{f'(z)}=\frac{2(1-\beta)zq'(z)}{(1-\beta)q(z)+\beta}+p.\]
Define the function $\psi:D\to\mathbb{C}$ where $D=\mathbb{C}^2\setminus\{(-\beta/(1-\beta),0)\}$ by
$$\psi(r,s)=\frac{2(1-\beta)s}{(1-\beta)r+\beta}+p,$$
and $\Omega=\{w\in\mathbb{C}:\RE w>\alpha\}$. Using the hypothesis, we have $ \psi(q(z),zq'(z))\in \Omega$ for all $z\in \mathbb{D}$. To prove that $\RE q(z)>0$, we apply Lemma \ref{MM}. Let $\rho\in\mathbb{R}$ and $\sigma \leq -(1+\rho^2)/2$. Then
\begin{align*}
\RE\psi(i\rho,\sigma)&=\frac{2\beta(1-\beta)\sigma}{(1-\beta)^2\rho^2+\beta^2}+p\\
&\leq \frac{-\beta(1-\beta)(1+\rho^2)}{(1-\beta)^2\rho^2+\beta^2}+p\\
&=-\beta(1-\beta)\phi(\rho^2,1,\beta) + p
\end{align*}
where $\phi$ is defined by \eqref{eq}. By Lemma \ref{lem}, it follows that
\[\min_{\rho\in \mathbb{R}} \phi(\rho^2,1,\beta)=\begin{cases}
	\displaystyle\frac{1}{(1-\beta)^2}, & 0 \leq\beta\leq\displaystyle\frac{1}{2},\\
	\\
	\displaystyle\frac{1}{\beta^2}, & \displaystyle\frac{1}{2}\leq \beta<1.
	\end{cases}\]
Therefore, it follows that
\[\RE\psi(i\rho,\sigma)\leq -\beta(1-\beta)\min_{\rho\in\mathbb{R}} \phi(\rho^2,1,\beta)+p= \begin{cases}
	p-\displaystyle\frac{\beta}{1-\beta}, & 0 \leq\beta\leq \displaystyle\frac{1}{2},\\
	\\
	p-\displaystyle\frac{1-\beta}{\beta}, & \displaystyle\frac{1}{2}\leq \beta<1.
	\end{cases}\]
Hence $\RE\psi(i\rho,\sigma)\leq \alpha$ so that $\psi(i\rho,\sigma)\notin \Omega$ and Lemma \ref{MM} gives the desired result.
\end{proof}

\begin{remark}
If $p=1$ and $\beta=1/2$, then Theorem \ref{th2.2} reduces to the first differential implication of \eqref{eq3}.
\end{remark}

The next two theorems concern the second differential implication of \eqref{eq4} for multivalent functions for specific values of $\gamma$. Given the complete range of $\gamma$, $0<\gamma<1$, the problem of finding $\beta$ is still unsolved.
\begin{theorem}\label{th2.4}
Let $p\in\mathbb{N}$ and $\gamma$ satisfies
\[\frac{p^2+1}{(p+1)^2}<\gamma<1.\]
If $f\in\mathcal{A}_{p}$ is locally $p$-valent and $$\RE\sqrt{\frac{f'(z)}{pz^{p-1}}}>\sqrt{\frac{(2p+1)\gamma-1}{2p}}\quad (z\in \mathbb{D}),$$ then $$\RE \left(\frac{f(z)}{z^{p}}\right)>\gamma \quad (z\in \mathbb{D}).$$
\end{theorem}

\begin{proof}
If we define the function $q:\mathbb{D}\to \mathbb{C}$ by
\[q(z)=\frac{1}{1-\gamma}\left(\frac{f(z)}{z^p}-\gamma\right)=1+\frac{a_{p+1}}{1-\gamma}z+\cdots,\]
then $q\in \mathcal{H}[1,1]$ and satisfies
\[\sqrt{\frac{f'(z)}{pz^{p-1}}}=\sqrt{(1-\gamma)q(z)+\gamma+(1-\gamma)\frac{zq'(z)}{p}}\]
for all $z\in \mathbb{D}$. Define the function $\psi:\mathbb{C}^2\to\mathbb{C}$ by
\[\psi (r,s)=\sqrt{(1-\gamma)r+\gamma+(1-\gamma)\frac{s}{p}}.\]
If we let $\rho\in\mathbb{R}$ and $\sigma \leq -(1+\rho^2)/2$, then $\psi(i\rho,\sigma):=\sqrt{\zeta}=\sqrt{\xi+i\eta}$ where $\xi=\gamma+(1-\gamma)\sigma/p$ and $\eta=(1-\gamma)\rho$. This gives
\begin{align*}
\xi=\gamma+(1-\gamma)\frac{\sigma}{p}&\leq \gamma-\frac{1-\gamma}{2p}(1+\rho^2)\\
&= \gamma-\frac{1-\gamma}{2p}\left(1+\frac{\eta^2}{(1-\gamma)^2}\right)\\
&=\gamma-\frac{(1-\gamma)^2+\eta^2}{2p(1-\gamma)}.
\end{align*}
A direct computation leads to
\[\xi+\sqrt{\xi^2+\eta^2}\leq \gamma-\frac{(1-\gamma)^2+\eta^2}{2p(1-\gamma)}+\sqrt{\left(\gamma-\frac{(1-\gamma)^2+\eta^2}{2p(1-\gamma)}\right)^2+\eta^2}:=h(\eta).\]
To find the maximum value of function $h$, we shall apply the second derivative test. To see this, observe that
\[h'(\eta)=-\frac{\eta}{p(1-\gamma)}+\frac{\eta((1-\gamma)(1-(2p+1)\gamma+2p^2(1-\gamma))+\eta^2)}{2p^2(1-\gamma)^2\sqrt{\left(\gamma-\frac{(1-\gamma)^2+\eta^2}{2p(1-\gamma)}\right)^2+\eta^2}}\]
so that $h'(0)=0$ and since \[\gamma>\frac{p^2+1}{(p+1)^2}>\frac{1}{2p+1}\]
we must have $(2p+1)\gamma-1>0$ and
\[h''(0)=\frac{2(1 - (2 p + 1)\gamma + p^2 (1 - \gamma))}{p(1-\gamma)((2 p + 1)\gamma-1)}.\]
Using the bounds on $\gamma$, one can see that the numerator of $h''(0)$ is negative while the  denominator of $h''(0)$ is positive and hence $h''(0)<0$. Thus the function $h$ attains its maximum at $\eta=0$ and
\[\xi+\sqrt{\xi^2+\eta^2}\leq h(\eta)\leq h(0)=\frac{(2p+1)\gamma-1}{p}\]
which implies that
\[\RE\psi(i\rho,\sigma)=\RE \sqrt{\zeta}=\sqrt{\frac{\xi+\sqrt{\xi^2+\eta^2}}{2}}\leq \sqrt{\frac{(2p+1)\gamma-1}{2p}}.\]
By Lemma \ref{MM}, it follows that $\RE q(z)>0$ for all $z\in \mathbb{D}$.
\end{proof}

\begin{theorem}\label{th2.5}
Let $p\in\mathbb{N}$. If $f\in\mathcal{A}_{p}$ is locally $p$-valent and $$\RE\sqrt{\frac{f'(z)}{pz^{p-1}}}>\frac{\sqrt{p}}{2}\quad (z\in \mathbb{D}),$$ then $$\RE \left(\frac{f(z)}{z^{p}}\right)>\frac{1}{2} \quad (z\in \mathbb{D}).$$
\end{theorem}

\begin{proof}
Define the function $q:\mathbb{D}\to\mathbb{C}$ by
\[q(z)=\frac{2f(z)}{z^{p}}-1=1+2a_{p+1}z+\cdots\]
A direct computation gives
\[\sqrt{\frac{f'(z)}{pz^{p-1}}}=\sqrt{\frac{1}{2}\left(q(z)+\frac{zq'(z)}{p}+1\right)}.\]
The function $\psi:\mathbb{C}^2\to\mathbb{C}$ defined by
$$\psi(r,s)=\sqrt{\frac{1}{2}\left(r+\frac{s}{p}+1\right)}$$
satisfies $ \RE \psi(q(z),zq'(z))>\sqrt{p}/2$ for all $z\in \mathbb{D}$. To apply Lemma \ref{MM}, let $\rho\in\mathbb{R}$ and $\sigma \leq -(1+\rho^2)/2$. Then
\[\psi(i\rho,\sigma)=\sqrt{\frac{1}{2}\left(i\rho+\frac{\sigma}{p}+1\right)}:=\sqrt{\zeta}=\sqrt{\xi+i\eta}\]
where $\xi=(\sigma+p)/(2p)$ and $\eta=\rho/2$. Using the conditions on $\rho$ and $\sigma$, we have
\[\xi=\frac{1}{2}\left(\frac{\sigma}{p}+1\right)\leq \frac{1}{2}\left(-\frac{1}{2p}(1+\rho^2)+1\right)=\frac{1}{2}\left(-\frac{1}{2p}(1+4\eta^2)+1\right)\]
A lengthy calculation shows that
\[\xi^2+\eta^2\leq\frac{(1+4\eta^2)((2p-1)^2+4\eta^2)}{16p^2}\]
so that
\[\sqrt{\xi^2+\eta^2}\leq \frac{1}{4p}\sqrt{(1+4\eta^2)(1-4p+4p^2+4\eta^2)}\leq \frac{1-2p+2p^2+4\eta^2}{4p}\]
using the relation between geometric and arithmetic mean. Now
\begin{align*}
\RE\psi(i\rho,\sigma)&=\RE \sqrt{\zeta}=\sqrt{\frac{\xi+\sqrt{\xi^2+\eta^2}}{2}}\\
   &\leq\sqrt{\frac{1}{2}\left(\frac{2p-1-4\eta^2}{4p}+\frac{1-2p+2p^2+4\eta^2}{4p}\right)}=\frac{\sqrt{p}}{2}
\end{align*}
Therefore the proof is completed by invoking Lemma \ref{MM}.
\end{proof}

\begin{remark}
The second differential implication of \eqref{eq3} corresponds to the case $p=1$ of Theorem \ref{th2.5}.
\end{remark}

 \section{Marx-Strohh\"{a}cker Type Results With Fixed Coefficient}
 This section is devoted to prove the Marx-Strohh\"{a}cker results \eqref{eq2} and \eqref{eq4} for functions with fixed coefficient. As pointed out earlier, Selvaraj and Stelin \cite{SELVARAJ} proved these results under milder conditions. They used Lemma \ref{ALI} inappropriately, in particular, the admissibility condition involving $\sigma$ and $\rho$ was wrongly written which led to incorrect results. The first theorem of this section generalizes \cite[Theorem 1, p.\ 355]{NUNO2} for functions with fixed coefficient.

 \begin{theorem}\label{thrm1}
 Let $p\in\mathbb{N}$, $0\leq \beta<p$ and $0\leq b\leq 2(p-\beta).$ Let $$\alpha=\alpha(\beta,p,b):=\begin{cases}
 		\beta\left(1-\displaystyle\frac{2}{2(p-\beta)+b}\right), & 0\leq\beta<\displaystyle\frac{p}{2},\\
 		\\
 		\beta\left(1-\displaystyle\frac{2(p-\beta)^2}{\beta^2(2(p-\beta)+b)}\right), & \displaystyle\frac{p}{2}\leq \beta<p.
 		\end{cases}$$
If $f\in\mathcal{A}_{p,b}$ satisfies $$\RE\left(1+\frac{zf''(z)}{f'(z)}\right)>\alpha\quad (z\in\disc),$$ then $$\RE\left(\frac{zf'(z)}{f(z)}\right)>\beta \quad (z\in\disc).$$
 	\end{theorem}
 \begin{proof}
 Define a function $q:\mathbb{D}\to \mathbb{C}$ by $$q(z)=\frac{1}{p-\beta}\left(\frac{zf'(z)}{f(z)}-\beta\right)=1+\frac{b}{p-\beta}z+\cdots.$$
Then $q\in\mathcal{H}_{b/(p-\beta)}[1,1]$. 	By using the equation $zf'(z)/f(z)=(p-\beta)q(z)+\beta$, we get $$1+\frac{zf''(z)}{f'(z)}=\frac{(p-\beta)zq'(z)}{(p-\beta)q(z)+\beta}+(p-\beta)q(z)+\beta.$$
 Let us define the function $\psi:\mathbb{C}^2\setminus\{(-\beta/(p-\beta), 0)\} \to\mathbb{C}$ by $$\psi(r,s)=\frac{(p-\beta)s}{(p-\beta)r+\beta}+(p-\beta)r+\beta.$$  Then $ \psi(q(z),zq'(z))\in \Omega$ for all $z\in \mathbb{D}$, where $\Omega=\{w\in\mathbb{C}:\RE w>\alpha\}$. For  $\rho \in \mathbb{R}$ and
 \[\sigma\leq -\frac{2(p-\beta)}{2(p-\beta)+b}(1+\rho^2),\]
we have
\begin{align*}
\RE\psi(i\rho,\sigma)&=\beta\left(\frac{(p-\beta)\sigma}{\beta^2+(p-\beta)^2\rho^2}+1\right)\\
 	&\leq \beta\left(\frac{-2(p-\beta)^2(1+\rho^2)}{(\beta^2+(p-\beta)^2\rho^2)(2(p-\beta)+b)}+1\right)\\
 &= \beta\left(\frac{-2(p-\beta)^2}{2(p-\beta)+b}\phi(\rho^2,p,\beta)+1\right)
 	\end{align*}
 where $\phi$ is defined by \eqref{eq}. Lemma \ref{lem} shows that
 \[\min_{\rho\in\mathbb{R}} \phi(\rho^2,p,\beta)=\begin{cases}
	\displaystyle\frac{1}{(p-\beta)^2}, & 0 \leq\beta\leq \displaystyle\frac{p}{2},\\
	\\
	\displaystyle\frac{1}{\beta^2}, & \displaystyle\frac{p}{2}\leq \beta<p.
	\end{cases}\]
A calculation using this bound shows that
\[\RE\psi(i\rho,\sigma)\leq\beta\left(1-\frac{2}{2(p-\beta)+b}\right)\]
if $0\leq \beta<p/2$ and
  \[\RE\psi(i\rho,\sigma)\leq \beta\left(1-\frac{2(p-\beta)^2}{\beta^2(2(p-\beta)+b)}\right)\]
if $p/2\leq \beta <p$.
Therefore, we conclude that $\RE\psi(i\rho,\sigma)\leq \alpha$ and therefore by Lemma \ref{ALI}, $\RE q(z)>0$ or $\RE (zf'(z)/f(z))>\beta$ for all $z\in \mathbb{D}$.
 \end{proof}

For $b=2(p-\beta)$, Theorem \ref{thrm1} reduces to \cite[Theorem 1, p.\ 355]{NUNO2}. If $p=1$ and $\beta =1/2$, then Theorem \ref{thrm2} yields \cite[Theorem 2.2, p.\ 228]{NAGPAL}. The correct version of \cite[Theorem 2.1, p.\ 886]{SELVARAJ} is contained in the following corollary.

\begin{corollary}
If $f\in\mathcal{A}_{p,b}$, $0\leq b\leq p$ satisfies $$\RE\left(1+\frac{zf''(z)}{f'(z)}\right)>\frac{p}{2}\left(1-\frac{2}{p+b}\right)\quad (z\in\disc),$$ then $$\RE\left(\frac{zf'(z)}{f(z)}\right)>\frac{p}{2} \quad (z\in\disc).$$
\end{corollary}

The next theorem determines a lower bound of $\RE (f(z)/z^p)$ over the unit disk $\disc$ for functions $f\in\mathcal{A}_{p,b}$ that are $p$-valent starlike of order $\beta$.
\begin{theorem}\label{thrm2}
Let $p\in\mathbb{N}$, $0<\gamma<1$ and $0\leq b\leq 2(1-\gamma)$. Also, let $$\beta=\beta(\gamma,p,b)=\begin{cases}
	p-\displaystyle\frac{2\gamma}{2(1-\gamma)+b},& 0<\gamma\leq\displaystyle\frac{1}{2},\\
	\\
	p-\displaystyle\frac{2(1-\gamma)^2}{\gamma(2(1-\gamma)+b)},& \displaystyle\frac{1}{2}\leq \gamma<1.
	\end{cases}$$
If the function $f\in\mathcal{A}_{p,b}$ satisfies $$\RE\left(\frac{zf'(z)}{f(z)}\right)>\beta\quad (z\in \mathbb{D}),$$ then $$\RE\left(\frac{f(z)}{z^p}\right)>\gamma\quad (z\in \mathbb{D}).$$
\end{theorem}
\begin{proof}
The function $q:\mathbb{D}\to \mathbb{C}$ defined by $$q(z)=\frac{1}{1-\gamma}\left(\frac{f(z)}{z^p}-\gamma\right)=1+\frac{b}{1-\gamma}z+\cdots$$ belongs to the class $\mathcal{H}_{b/(1-\gamma)}[1,1].$ The equation shows that $f(z)=((1-\gamma)q(z)+\gamma)z^p$ gives $$\frac{zf'(z)}{f(z)}=\frac{(1-\gamma)zq'(z)}{(1-\gamma)q(z)+\gamma}+p.$$
	If we define the function $\psi:D\to\mathbb{C}$ where $D=\mathbb{C}^2\setminus\{(-\gamma/(1-\gamma),0)\}$ by $$\psi(r,s)=\frac{(1-\gamma)s}{(1-\gamma)r+\gamma}+p,$$
then it satisfies $\RE\psi(q(z),zq'(z))>\beta$ for $z\in \mathbb{D}$. To apply Lemma \ref{ALI}, note that
	\begin{align*}
	\RE\psi(i\rho,\sigma) &=\frac{\gamma(1-\gamma)\sigma}{(1-\gamma)^2\rho^2+\gamma^2}+p\\
	&\leq \frac{-2\gamma(1-\gamma)^2(1+\rho^2)}{(2(1-\gamma)+b)((1-\gamma)^2\rho^2+\gamma^2)}+p\\
&= \frac{-2\gamma(1-\gamma)^2}{(2(1-\gamma)+b)}\phi(\rho^2,1,\gamma)+p
	\end{align*}
where $\phi$ is defined by \eqref{eq}, $\rho\in \mathbb{R}$ and
\[\sigma\leq -\frac{2(1-\gamma)}{(2(1-\gamma)+b)}(1+\rho^2).\]
By Lemma \ref{lem}, it is easy to see that
\[\min_{\rho\in\mathbb{R}} \phi(\rho^2,1,\gamma)=\begin{cases}
	\displaystyle\frac{1}{(1-\gamma)^2}, & 0 <\gamma\leq\displaystyle\frac{1}{2},\\
	\\
	\displaystyle\frac{1}{\gamma^2}, & \displaystyle\frac{1}{2}\leq \gamma<1.
	\end{cases}\]
so that
\begin{align*}
\RE\psi(i\rho,\sigma)&\leq \frac{-2\gamma(1-\gamma)^2}{(2(1-\gamma)+b)}\min_{\rho\in\mathbb{R}} \phi(\rho^2,1,\gamma)+p\\
&=\begin{cases}
	p-\displaystyle\frac{2\gamma}{2(1-\gamma)+b},& 0<\gamma\leq\displaystyle\frac{1}{2},\\
	\\
	p-\displaystyle\frac{2(1-\gamma)^2}{\gamma(2(1-\gamma)+b)},& \displaystyle\frac{1}{2}\leq \gamma<1.
	\end{cases}
\end{align*}
Hence $\RE (\psi(i\rho,\sigma))\leq \beta$ and therefore $\RE q(z)>0$ for all $z\in \mathbb{D}$ by Lemma \ref{ALI}.
\end{proof}

Theorem \ref{thrm2} reduces to \cite[Theorem 3, p.\ 359]{NUNO2} for $b=2(1-\gamma)$. If $p=1$ and $\beta=1/2,$ then Theorem \ref{thrm2}  reduces to \cite[Theorem 2.6, p.\ 230]{NAGPAL}. The following corollary gives the correct form of \cite[Theorem 2.8, p.\ 890]{SELVARAJ}.

\begin{corollary}
If $f\in\mathcal{A}_{p,b}$, $0\leq b\leq 1$ satisfies $$\RE\left(\frac{zf'(z)}{f(z)}\right)>p-\frac{1}{1+b}\quad (z\in \mathbb{D}),$$ then $$\RE\left(\frac{f(z)}{z^p}\right)>\frac{1}{2}\quad (z\in \mathbb{D}).$$
\end{corollary}

The last three theorems generalize the results obtained in Section 2. This in turn gives correct versions of the corresponding results of \cite{SELVARAJ} as well.

\begin{theorem}\label{thrm3}
Let $p\in\mathbb{N},$ $0\leq \beta<1$ and $0\leq (p+1)b\leq 4p(1-\beta).$ Let
	\begin{align}
	\alpha=\alpha(\beta,p,b)= \begin{cases}
	p-\displaystyle\frac{8p\beta}{4p(1-\beta)+(p+1)b},& 0\leq \beta\leq\displaystyle\frac{1}{2}\\
	\\
	p-\displaystyle\frac{8p(1-\beta)^2}{\beta(4p(1-\beta)+(p+1)b)},& \displaystyle\frac{1}{2}\leq\beta <1.
	\end{cases}
	\end{align}
If the function $f\in\mathcal{A}_{p,b}$ satisfies $$\RE\left(1+\frac{zf''(z)}{f'(z)}\right)>\alpha\quad (z\in \mathbb{D}),$$ then $$\RE\sqrt{\frac{f'(z)}{pz^{p-1}}}>\beta \quad(z\in \mathbb{D}).$$
\end{theorem}
\begin{proof}
Using the same definition and notation for the functions $q$ and $\psi$ as defined in Theorem \ref{th2.2}, it is evident that
\[q(z)=1+\frac{(p+1)b}{2p(1-\beta)}z+\cdots \quad (z\in \mathbb{D}),\]
belongs to the class $\mathcal{H}_{\frac{(p+1)b}{2p(1-\beta)}}[1,1].$ For $\rho\in\mathbb{R}$ and
\[\sigma\leq -\frac{4p(1-\beta)}{4p(1-\beta)+(p+1)b}(1+\rho^2),\]
a straight forward calculation shows that
	\begin{align*}
\RE\psi(i\rho,\sigma)&=\frac{2\beta(1-\beta)\sigma}{(1-\beta)^2\rho^2+\beta^2}+p\\
&\leq \frac{-8p\beta(1-\beta)^2(1+\rho^2)}{(4p(1-\beta)+(p+1)b)(1-\beta)^2\rho^2+\beta^2)}+p\\
&=\frac{-8p\beta(1-\beta)^2}{4p(1-\beta)+(p+1)b}\phi(\rho^2,1,\beta)+p
\end{align*}
where $\phi$ is defined by \eqref{eq}. Using Lemmas \ref{lem} and \ref{ALI}, a similar calculation gives the desired result.
\end{proof}

If $(p+1)b= 4p(1-\beta)$, then Theorem \ref{thrm3} reduces to Theorem \ref{th2.2}.  Also, the case $\beta=1/2$ gives the following result which corrects \cite[Theorem 2.5, p.\ 888]{SELVARAJ}.

\begin{corollary}
If $f\in\mathcal{A}_{p,b}$, $0\leq (p+1)b\leq 2p$ satisfies $$\RE\left(1+\frac{zf''(z)}{f'(z)}\right)>p-\frac{4p}{2p+(p+1)b}\quad(z\in \mathbb{D}),$$ then $$\RE\sqrt{\frac{f'(z)}{pz^{p-1}}}>\frac{1}{2}\quad(z\in \mathbb{D}).$$
\end{corollary}

\begin{theorem}\label{th3.7}
Let $p\in\mathbb{N}$, $\gamma\in (0,1)$ such that $0\leq b\leq 2(1-\gamma)$ and
\begin{equation}\label{eq3.2}
p^2(2(1-\gamma)+b)^2+16(1-\gamma)^2-8p\gamma(2(1-\gamma)+b)<0.
\end{equation}
If $f\in\mathcal{A}_{p,b}$ is locally $p$-valent and
$$\RE\sqrt{\frac{f'(z)}{pz^{p-1}}}>\sqrt{\frac{(4+(2+b)p)\gamma-2(1+p)\gamma^2-2}{p(b+2(1-\gamma))}}:=\beta\quad (z\in \mathbb{D}),$$ then $$\RE \left(\frac{f(z)}{z^{p}}\right)>\gamma \quad (z\in \mathbb{D}).$$
\end{theorem}

\begin{proof}
With the notations as defined in Theorem \ref{th2.4}, it is easily seen that $q \in \mathcal{H}_{\frac{b}{1-\gamma}}[1,1]$ and the theorem is proved if the admissibility condition $\RE \psi(i\rho,\sigma)\not\in \Omega$ whenever $\rho\in \mathbb{R}$ and
\[\sigma\leq -\frac{2(1-\gamma)}{2(1-\gamma)+b}(1+\rho^2)\]
is satisfied, where $\Omega=\{w:\RE w>\beta\}$. On similar lines, we obtain
\begin{align*}
\xi&\leq \gamma-\frac{2(1-\gamma)^2}{(2(1-\gamma)+b)p}\left(1+\frac{\eta^2}{(1-\gamma)^2}\right)\\
&=\frac{(4+(2+b)p)\gamma-2(1+p)\gamma^2-2(1+\eta^2)}{p(2(1-\gamma)+b)}
\end{align*}
so that
\begin{align*}
\xi+\sqrt{\xi^2+\eta^2}&\leq \frac{(4+(2+b)p)\gamma-2(1+p)\gamma^2-2(1+\eta^2)}{p(2(1-\gamma)+b)}\\
&\qquad\qquad+\sqrt{\left(\frac{(4+(2+b)p)\gamma-2(1+p)\gamma^2-2(1+\eta^2)}{p(2(1-\gamma)+b)}\right)^2+\eta^2}:=g(\eta).
\end{align*}
Observe that $g'(0)=0$ and
\begin{align*}
g''(0)&=\frac{p^2(2(1-\gamma)+b)^2+16(1-\gamma)^2-8p\gamma(2(1-\gamma)+b)}{p(2(1-\gamma)+b)((4+(2+b)p)\gamma-2(1+p)\gamma^2-2)}\\
      &=\frac{2(p+2)^2(\gamma-\gamma_3)(\gamma-\gamma_4)}{p(p+1)(2(1-\gamma)+b)(\gamma-\gamma_1)(\gamma_2-\gamma)}
\end{align*}
where $\gamma_i$ ($i=1,2,3,4$) are defined by
\[\gamma_1=\frac{4+(2+b)p-\sqrt{p}\sqrt{b(8+bp)+4p(1+b)}}{4(1+p)},\]
\[\gamma_2=\frac{4+(2+b)p+\sqrt{p}\sqrt{b(8+bp)+4p(1+b)}}{4(1+p)}\]
\[\gamma_3=\frac{8+p(2+b)(2+p)-4\sqrt{p^2(1+b)+2pb}}{2(p+2)^2},\]
\[\gamma_4=\frac{8+p(2+b)(2+p)+4\sqrt{p^2(1+b)+2pb}}{2(p+2)^2}.\]
Using the hypothesis, we see that the numerator of $g''(0)$ is negative. Also, since $\gamma_4\geq 1$ and $\gamma\in (0,1)$, therefore \eqref{eq3.2} may be re-written as $\gamma_3<\gamma<1$. In order to apply second derivative test, we need to show that $g''(0)<0$ or equivalently the denominator of $g''(0)$ is positive.

Note that $\gamma_2\geq 1$ so that $\gamma<1\leq \gamma_2$ which implies that $\gamma_2-\gamma>0$. It remains to show that $\gamma-\gamma_1>0$. Observe that
\[(\gamma_3-\gamma_1)(\gamma_2-\gamma_3)=\frac{p^2q(b,p)}{4 (1 + p) (2 + p)^4}>0\]
where $q(b,p)=b^2 (2 + p)^2 + 8 p (p + b (2 + p)) +4 (2 (p + b) + b p) \sqrt{p} \sqrt{p + b (2 + p)}$. Since $\gamma_2\geq 1>\gamma_3$, it follows that $\gamma_3-\gamma_1>0$ so that $ \gamma_1<\gamma_3<\gamma$ and hence $\gamma-\gamma_1>0$. Consequently the function $g$ attains its maximum at $\eta=0$. Thus
\[\RE\psi(i\rho,\sigma)=\sqrt{\frac{\xi+\sqrt{\xi^2+\eta^2}}{2}}\leq \sqrt{\frac{g(0)}{2}}\leq \sqrt{\frac{(4+(2+b)p)\gamma-2(1+p)\gamma^2-2}{p(b+2(1-\gamma))}}.\]
Lemma \ref{MM} gives the desired result.
\end{proof}

\begin{theorem}\label{th3.8}
Let $p\in\mathbb{N}$ and $0\leq b\leq 1$. If the function $f\in\mathcal{A}_{p,b}$ is locally $p$-valent and $$\RE\sqrt{\frac{f'(z)}{pz^{p-1}}}>\sqrt{\frac{(1+b)p}{8}}\quad (z\in \mathbb{D}),$$ then $$\RE \left(\frac{f(z)}{z^{p}}\right)>\frac{1}{2} \quad (z\in \mathbb{D}).$$
\end{theorem}
\begin{proof}
Proceeding as in the proof of Theorem \ref{th2.5} with $q$, $\psi$, $\zeta$, $\eta$ and $\xi$, we need to show that the admissibility condition
\[\RE\psi(i\rho,\sigma)\leq \sqrt{\frac{(1+b)p}{8}}\]
is satisfied whenever $\rho\in\mathbb{R}$ and $\sigma \leq -(1+\rho^2)/(1+b)$. Note that
\[\xi=\frac{1}{2}\left(\frac{\sigma}{p}+1\right)\leq \frac{1}{2}\left(-\frac{1}{p(1+b)}(1+4\eta^2)+1\right)\]
A tedious calculation shows that
\[\xi^2+\eta^2\leq \frac{(1+4\eta^2)(1 - 2 (1 + b) p + (1 + b)^2 p^2 + 4 \eta^2)}{4 (1 + b)^2 p^2}\]
so that
\begin{align*}
\sqrt{\xi^2+\eta^2}&\leq \frac{1}{2(1+b)p}\sqrt{(1+4\eta^2)(1 - 2 (1 + b) p + (1 + b)^2 p^2 + 4 \eta^2)}\\
&\leq \frac{2 - 2 (1 + b) p + (1 + b)^2 p^2 + 8 \eta^2}{4(1+b)p}
\end{align*}
and a computation gives
\[\RE\sqrt{\zeta}=\sqrt{\frac{\xi+\sqrt{\xi^2+\eta^2}}{2}}
   \leq \sqrt{\frac{(1+b)p}{8}} \qedhere\]
\end{proof}

Note that the results stated in Theorems \ref{th3.7} and \ref{th3.8} coincides with Theorems \ref{th2.4} and \ref{th2.5} if $b=2(1-\gamma)$ and $b=1$ respectively. Also, observe that although \cite[Theorem 2.11, p.\ 891]{SELVARAJ} has produced the same result as that in Theorem \ref{th3.8}, there is a calculation mistake in the proof of that theorem.

\section*{Acknowledgements}
The research work of first author presented here is supported by a grant from Council of Scientific and Industrial Research (CSIR), New Delhi.

 \end{document}